\documentclass[letterpaper, 10 pt, conference]{ieeeconf}  

\usepackage{cite}
\usepackage{amsmath,amssymb,amsfonts}
\usepackage{mathtools}
\usepackage{graphicx}
\graphicspath{{photos/}}
\usepackage{color}
\usepackage{enumerate}   
\usepackage{algpseudocode}
\usepackage{algorithm}
\usepackage{siunitx}
\usepackage{bbold}
\usepackage{multicol}
\usepackage{booktabs}
 \usepackage{multirow} 
\usepackage{hyperref}
\usepackage{textcomp}
\usepackage{subfigure}
\usepackage{stfloats}
\usepackage{epsfig}
\usepackage{xcolor}

\newcommand{\Nset}{\mathbb{N}}
\newcommand{\Rset}{\mathbb{R}}

\newcommand{\Uset}{\mathbb{U}}

\newcommand{\Yset}{\mathbb{Y}}

\newcommand{\bg}{\mathbf{g}}

\definecolor{Matlab_blue}{rgb}{0, 0.4470, 0.7410} 
\definecolor{Matlab_orange}{rgb}{0.8500, 0.3250, 0.0980} 
\definecolor{Matlab_yellow}{rgb}{0.9290 0.6940 0.1250} 
\definecolor{green_dark}{rgb}{0.0, 0.5, 0.0}

\newcommand{\col}{\operatorname{col}}

\newcommand{\figurename}[1]{Fig.#1}

\newtheorem{theorem}{Theorem}[section]
\newtheorem{problem}[theorem]{Problem}
\newtheorem{lemma}[theorem]{Lemma}

\newtheorem{definition}[theorem]{Definition}
\newtheorem{remark}[theorem]{Remark}
\newtheorem{assumption}[theorem]{Assumption}

\IEEEoverridecommandlockouts                              % This command is only needed if 
\title{\LARGE \bf
A Kernelized Operator Approach to \\Nonlinear Data-Enabled Predictive Control
}

\author{
Thomas de Jong$^{1}$, Siep Weiland$^{1}$,
Mircea Lazar$^{1}$
\thanks{$^{1}$~Department of Electrical Engineering,
        Eindhoven University of Technology, 5612 AZ Eindhoven, The Netherlands. E-mails of the authors: \texttt{t.o.d.jong@tue.nl, S.Weiland@tue.nl, m.lazar@tue.nl}.}
}

\begin{document}

\maketitle
\thispagestyle{empty}
\pagestyle{empty}

%%%%%%%%%%%%%%%%%%%%%%%%%%%%%%%%%%%%%%%%%%%%%%%%%%%%%%%%%%%%%%%%%%%%%%%%%%%%%%%%
\begin{abstract}
This paper considers the design of nonlinear data-enabled predictive control (DeePC) using kernel functions. Compared with existing methods that use kernels to parameterize multi-step predictors for nonlinear DeePC, we adopt a novel, operator-based approach. More specifically, we employ a universal product kernel parameterization of nonlinear systems operators as a prediction mechanism for nonlinear DeePC. 
We show that by using a product reproducing kernel Hilbert space (RKHS) to learn the system trajectories, big data sets can be handled effectively to construct the corresponding product Gram matrix. Moreover, we show that the structure of the adopted product RKHS representation allows for a computationally efficient DeePC formulation. Compared to existing methods, our approach achieves substantially faster computation times for the same data size. This allows for the use of much larger data sets and enhanced control performance. 
\end{abstract}

%%%%%%%%%%%%%%%%%%%%%%%%%%%%%%%%%%%%%%%%%%%%%%%%%%%%%%%%%%%%%%%%%%%%%%%%%%%%%%%%
\section{INTRODUCTION}
Willems' fundamental lemma underpins data-driven control by enabling the prediction of system behavior directly from data, without requiring an explicit model \cite{willems2005note}. This concept forms the foundation of data-enabled predictive control (DeePC), which uses historical data for predicting future trajectories \cite{coulson2019data}. For linear time-invariant (LTI) systems and noise-free data, equivalences between Model Predictive Control (MPC), Subspace Predictive Control (SPC) and DeePC have been established \cite{fiedler2021relationship}. Extending DeePC to nonlinear systems has been an area of active research. For example, in \cite{lazar2024basis}, DeePC was generalized to nonlinear systems using general basis functions, while in \cite{lian2021koopman}, the Koopman framework was employed to parameterize the DeePC predictor. The Koopman approach provides a way to represent nonlinear dynamics by mapping system states to a higher-dimensional space by means of observable functions, where a linear operator governs their evolution. Identifying suitable basis functions or Koopman observable functions can be done using for example neural networks as was done in \cite{lazar2024neural} to construct a basis, or as done in \cite{lian2021koopman,de2024koopman} for learning Koopman observable functions.

Alternatively, function estimation using regularized kernel methods offers a tractable solution for learning dynamical systems due to the well-posedness of function classes in RKHSs \cite{pillonetto2014kernel, scholkopf2001generalized}. This has led to increased interest in kernel-based approaches for data-driven control \cite{care2023kernel, martin2023guarantees, hu2023learning}. Several kernel-based approaches have also emerged in the field of data-driven predictive control. In \cite{maddalena2021kpc} a multi-step kernel-based predictor was identified and used within an MPC framework; therein, tight bounds on the kernelized multi-step predictors' approximation error were derived. More recently, \cite{huang2023robust} presented a kernelized DeePC formulation for nonlinear systems, which corresponds to solving a system of equations similar to Willem's fundamental lemma. Therein, methods for robustification and efficient online implementation based on kernels that are linear in future control inputs were presented. In \cite{lazar2024basis} it was shown that the special structure of basis or kernel functions that are linear in the future control inputs corresponds in fact to a DeePC formulation of linear-in-control input Koopman MPC \cite{korda2020koopman}. The above kernelized predictors for data-driven predictive control require in general usage of kernel functions defined on the stacked space of initial conditions and future control input sequences. Generating long persistently exciting input sequences that sample effectively the corresponding stacked space can be challenging for nonlinear systems, however. Persistency of excitation conditions that lead to exact kernelized representations for certain classes of nonlinear systems were recently presented in \cite{Timm_Ker_Lemma}, along with kernelized fundamental lemmas for linear, Hammerstein and flat nonlinear systems. 

Alternatively, a discrete-time dynamical system can be regarded as an operator defined on a Hilbert space \cite{van2022kernel} that maps sequences of inputs to sequences of outputs. More recently, \cite{Lazar2024productkernel} defined the domain of the operator as a \textit{product} Hilbert space. The corresponding operator maps input sequences and states to output sequences. Therein, it was proven that for positive definite kernels \cite{Micchelli2006} the resulting product RKHS is dense in the space of nonlinear systems operators and that its unique minimizer is a universal approximator. The developed results in \cite{Lazar2024productkernel} yield an alternative formulation of Willems' fundamental lemma and a new insight in persistency of excitation for nonlinear systems. Specifically, \cite{Lazar2024productkernel} points out that a that a persistently exciting input sequence should be paired with a linearly independent/distinct set of points in the space of initial conditions. This allows for generating larger data sets of trajectories by using a short, persistently exciting input sequence that is applied to several distinct initial conditions. Hence, it would be very interesting to employ the product RKHS operator parameterization for designing nonlinear DeePC algorithms. 

To this end, in this paper, building on the framework from \cite{Lazar2024productkernel}, we develop a kernelized operator approach for data-driven control of nonlinear systems through the DeePC framework. By modeling discrete-time nonlinear systems as operators on a Hilbert space, our method generalizes beyond linear systems and incorporates specific classes of nonlinear systems not considered by previous methods that rely directly on the fundamental lemma such as \cite{Timm_Ker_Lemma, lian2021nonlinear}. A main challenge in extending the kernelized operator framework of \cite{Lazar2024productkernel} for data-driven control is the dependence of the predictor on the Gram matrix, whose dimension scales with both input and state data. We demonstrate that the product kernel structure allows separation of the input and state data, which reduces the dimension of the DeePC equality constraints significantly. Additionally, we introduce a systematic data-generation method, providing a structured approach for creating datasets required to learn the operator. Our product kernel formulation offers computational advantages over previous kernel-based DeePC approaches, such as \cite{huang2023robust}.

%The remainder of the paper is organized as follows. After introducing the preliminaries in Section~\ref{sec:preliminaries_problem_statement}, we discuss the kernel based Data Enabled Predictive control problem and operator (KerODeePC) learning in Section~\ref{sec:KerODeePC}. In Section~\ref{sec:reduction} we present a computationally efficient implementation of the KerODeePC problem. Finally, in Section~\ref{sec:example}, we analyze the effectiveness of our proposed control scheme through a nonlinear example where we compare the performance and computational efficiency with an existing method. 

\paragraph*{Notation and basic definitions} Let $\mathbb{R}$, $\mathbb{R}_{+}$ and $\mathbb{N}$ denote the set and field of real numbers, the set of non-negative reals, and the set of non-negative integers. For every $c \in \mathbb{R}$ and $\Pi \subseteq \mathbb{R}$ define $\Pi_{\geq c} := \{k \in \Pi | k \geq c \}$. A square $n \times n$ matrix where all elements of  $1/n$ is denoted by $\mathbb{1}_{n}$ and $\mathbf{1}_n$ denotes a vector in $\mathbb{R}^n$ with all elements equal to $1/n$. For two matrices $A \in \Rset^{n\times m}$, $B \in \Rset^{p \times q}$, we denote their Kronecker product by $A\otimes B$, defined as
$$
A \otimes B = \begin{bmatrix}
    a_{1,1} B & \dots & a_{1,m} B\\
    \vdots & & \vdots \\
    a_{n,1} B & \dots & a_{n,m} B
\end{bmatrix}.
$$
Throughout this paper, for any finite number $q\in\mathbb{N}_{\geq 1}$ of column vectors or functions $\{\xi_1,\dots,\xi_q\}$ we will make use of the operator $\col(\xi_1,\dots,\xi_q) := [\xi_1^\top,\dots,\xi_q^\top]^\top$.

\section{Preliminaries}\label{sec:preliminaries_problem_statement}
We consider nonlinear MIMO systems in state space form with inputs $u\in\mathbb{R}^m$, states $x\in\mathbb{R}^n$ and outputs $y\in\mathbb{R}^p$, i.e.,
\begin{equation}\label{eqn:state_space}
    \begin{aligned}
         x_{k+1} &= f(x_{k},u_{k}), \quad k\in\mathbb{N}, \\
        y_{k} &= h(x_{k}).
    \end{aligned}
\end{equation}
The functions $f:\mathbb{R}^{n} \times \mathbb{R}^{m} \rightarrow \mathbb{R}^{n}$ and $h:\mathbb{R}^{n} \rightarrow \mathbb{R}^{p}$ are assumed to be \emph{unknown}. We denote by $\mathbb{X} \subseteq \mathbb{R}^n$, $\mathbb{U} \subseteq \mathbb{R}^m$ and $\mathbb{Y} \subseteq \mathbb{R}^p$ the constraint admissible sets of states, inputs and outputs respectively. 

For simplicity of exposition we assume that the state is measured, but the derived methods extend directly to a surrogate state consisting of stacked past inputs and outputs. 
\begin{definition}
    A function $k:\mathcal{Z} \times \mathcal{Z} \rightarrow \mathbb{R}$, where $\mathcal{Z}$ denotes the data space, is called a \emph{kernel function} if it satisfies the following properties: (\emph{i}) it is symmetric, i.e., $k(z_1, z_2) = k(z_2, z_1)$ for all $(z_1, z_2) \in \mathcal{Z} \times \mathcal{Z} $; and (\emph{ii}) it is positive semi-definite, i.e., for any $T>0$ and any $\{z_1, . . . , z_T \} \in \mathcal{Z}$ the matrix
    \begin{align}
        K := \begin{pmatrix}
            k(z_1,z_1) & \dots & k(z_{1},z_{T}) \\
            \vdots & \ddots & \vdots \\
            k(z_{T},z_1) & \dots & k(z_{T},z_{T}) 
            \end{pmatrix} \in \mathbb{R}^{T \times T},
    \end{align}    
    is positive semi-definite. The matrix $K$ is called the Gram matrix. Moreover, a kernel function $k$ is called a \emph{universal kernel} (or a positive definite kernel) if its corresponding Gram matrix is positive definite for any $\{z_1,\dots,z_T\}\in\mathcal{Z}$ with distinct points and any $T>0$.
\end{definition}

\begin{definition}
    Given a kernel function $k : \mathcal{Z} \times \mathcal{Z} \rightarrow \Rset$, a Hilbert space $\mathcal{H}(k, \mathcal{Z})$ is a reproducing kernel Hilbert space of functions from $\mathcal{Z}$ to $\mathbb{R}$ for $k$ if (\emph{i}) for every $z \in \mathcal{Z}$ , the function $k(z, \cdot) \in \mathcal{H}(k, \mathcal{X} )$ and (\emph{ii}) the reproducing property holds, i.e., $f(z) = \langle f, k(z,\cdot) \rangle_{\mathcal{H}(k,\mathcal{Z})}$ for every $f \in \mathcal{H}(k,\mathcal{Z})$ and $z \in \mathcal{Z}$, where $\langle \cdot, \cdot \rangle_{\mathcal{H}(k,\mathcal{Z})}$ denotes the inner product associated with $\mathcal{H}(k,\mathcal{Z})$. For a positive definite kernel $k$ we have that for all $T>0$ it holds that $\text{span}\{k(\cdot, z_1), k(\cdot, z_2), \dots , k(\cdot, z_T ) \}  \subseteq \mathcal{H}(k, \mathcal{Z})$.
\end{definition}

In what follows we will make use of the following notation for the vector-kernel corresponding to the standard basis, i.e., $\text{span}\{k(\cdot, z_1), k(\cdot, z_2), \dots , k(\cdot, z_T )\}:$
$$
 \mathbf{k}(z) := \col(k(z, z_1), k(z, z_2), \dots  , k(z, z_T )) \in \Rset^T, \forall z \in \mathcal{Z}.
$$
The entries of the Gram matrix or the kernel functions are related to a feature map as follows, 
$$
K_{ij} = k(z_i,z_j) = \langle \phi(z_i), \phi(z_j) \rangle_{\mathcal{H}(k,\mathcal{Z})},
$$
where $k(\cdot,\cdot)$ is a nonlinear kernel function and $\phi : \mathcal{Z} \rightarrow \mathcal{H}(k,\mathcal{Z})$ is the feature map. It is noteworthy to mention that the above equation does not explicitly require the computation of the feature map $\phi(z_j)$, but requires only the characterization of the dot product in the feature space defined using a kernel function $k$. Two important classes of universal kernels \cite{Micchelli2006}, which are also radial basis functions that comply with the conditions of \cite{chen1995approximation} are Gaussian kernels, i.e.,
\begin{equation}\label{eqn:gaussian}
    k(z_1,z_1) = e^{-\frac{1}{\sigma^2} \|z_1-z_2 \|_2^2}
\end{equation}
and Hardy reverse multiquadratics kernels, i.e.,
\begin{equation}\label{eqn:hardy}
    k(z_1,z_1) = \left( 1 + \frac{1}{\sigma^2} \|z_1 - z_2 \|_2^2 \right)^{\frac{1}{2}}
\end{equation}
where $\sigma$ is a positive real number. 

To establish the problem setup for some finite $N\in\Nset_{\geq 1}$ we define relevant Hilbert spaces of square summable sequences $\mathcal{U}:= l_2(\{0,1,\dots,N-1\},\Rset^m)$, $\mathcal{Y}:= l_2(\{1,\dots,N\},\Rset^p)$ and $\mathcal{X}:= l_2(\{0\},\Rset^n)$, where the notation was inspired by \cite{shali2024towards}. Next consider a set of input trajectories $\mathbf{U} = \{\mathbf{u}_1,\dots,\mathbf{u}_{T_u}\}$ with $\mathbf{u}_i \in \mathcal{U}$, a set of initial states $\mathbf{X}_0 = \{x_1,\dots,x_{T_x}\}$ with $x_j\in\mathcal{X}$ and a corresponding set of output trajectories $\mathbf{Y} = \{\mathbf{y}^1_1,\dots,\mathbf{y}^1_{T_u},\dots,\mathbf{y}^{T_x}_{1},\dots,\mathbf{y}^{T_x}_{T_u}\}$ with $\mathbf{y}^j_i \in \mathcal{Y}$. To this end let $k_u : \mathcal{U} \times \mathcal{U} \rightarrow \Rset$ and $k_x : \mathcal{X} \times \mathcal{X} \rightarrow \Rset$ be kernel functions, respectively. Let $K_u \in \mathbb{R}^{T_u\times T_u}$, $K_x \in \mathbb{R}^{T_x \times T_x}$ denote the corresponding Gram matrices and for any $x \in \mathcal{X}$, $\mathbf{u} \in \mathcal{U}$ and $\mathbf{y} \in \mathcal{Y}$ let  $\mathbf{k}_u(\mathbf{u})$, $\mathbf{k}_x(x)$ denote the corresponding vector-kernels, i.e., 
\begin{align*}
    \mathbf{k}_u(\mathbf{u}) := \begin{bmatrix}
        k_u(\mathbf{u},\mathbf{u}_1) \\
        \vdots \\
        k_u(\mathbf{u},\mathbf{u}_{T_u})
    \end{bmatrix}, \quad \mathbf{k}_x(x) := \begin{bmatrix}
        k_x(x,x_1) \\
        \vdots \\
        k_x(x,x_{T_u})
    \end{bmatrix}.
\end{align*}
Specifically, we build the product kernel function (e.g., see \cite{steinwart2008support}) by $k_{\otimes} : (\mathcal{U}\times \mathcal{X}) \times (\mathcal{U}\times \mathcal{X}) \rightarrow \mathbb{R}$  as introduced in \cite{Lazar2024productkernel}
\begin{align}
     k_{\otimes}((\mathbf{u}_1,x_1),(\mathbf{u}_2,x_2)) := k_u(\mathbf{u}_1,\mathbf{u}_2) k_x(x_1,x_2),
\end{align}
with corresponding Gram matrix and vector-kernel, i.e.,
\begin{align}
     K_{\otimes} = K_u \otimes K_x,\text{ and } \mathbf{k}_{\otimes}(\mathbf{u},x) = \mathbf{k}_u(\mathbf{u}) \otimes \mathbf{k}_x(x),
\end{align}
and unique product RKHS $\mathcal{H}(k_\otimes, \mathcal{X} \times \mathcal{U})$. As pointed out in \cite{Lazar2024productkernel}, through the prism of Willems’ fundamental lemma \cite{willems2005note}, we consider the following system of equations: 
\begin{subequations}
\label{eq:willems}
    \begin{align}
    K_{\otimes} \mathbf{g} &= \mathbf{k}_{\otimes}(\mathbf{u},x), \label{eqn:willemsa} \\
    \mathbf{Y} \mathbf{g} &= \mathbf{y}, 
    \end{align}
\end{subequations}
where $\mathbf{g} \in \mathbb{R}^{T_u T_x}$. 
\begin{remark}\label{kernelized_lemma} The kernelized formulation of the fundamental lemma from \cite{Timm_Ker_Lemma,huang2023robust} is recovered from \eqref{eq:willems} if $T_x=1$ and $k_x(x,x) = 1$ for all $x\in\mathcal{X}$. 

This yields the formulation: 
\begin{subequations}
\label{eq:willems_u}
    \begin{align}
   K_u \mathbf{g} &= \mathbf{k}_u(\mathbf{u}), \\
    \mathbf{Y} \mathbf{g} &= \mathbf{y}, 
    \end{align}
\end{subequations}
which corresponds to using a data set with only one initial condition. The conditions $k_x(x,x) = 1$ for all $x\in\mathcal{X}$ is satisfied for many popular kernels such as the Gaussian kernel \eqref{eqn:gaussian} and Hardy reverse multiquadratics kernels \eqref{eqn:hardy}. The fundamental lemma for discrete-time LTI systems \cite{willems2005note} is recovered from \eqref{eq:willems_u} if in addition linear kernel functions are used for $k_u$ as shown in \cite{Timm_Ker_Lemma}.
\end{remark}

\section{Kernelized Operator Data-Enabled Predictive Control}\label{sec:KerODeePC}
In this section we present the new formulation of nonlinear DeePC using product kernel-based operators. 
\subsection{Kernelized Operator DeePC problem formulation}
We would like to use the data sets $\mathbf{U}$, $\mathbf{Y}$ and $\mathbf{X}_0$ to learn the underlying operator $\mathbf{y} = G(\mathbf{u})(x)$ that maps pairs of points in $\mathcal{U}$ and $\mathcal{X}$ into points in $\mathcal{Y}$, i.e.,
$$
G(\mathbf{u})(x) = \Theta^* \mathbf{k}_{\otimes}(\mathbf{u},x),
$$
where $\Theta^*\in\mathbb{R}^{Np\times T_x T_u}$ is unknown. The operator learning problem can be formulated as
\begin{align}\label{eqn:operator_learning}
    \Theta^* = \arg \min_{\Theta} \| \mathbf{Y} -  \Theta K_{\otimes} \|_F^2,
\end{align}
where $\|M\|_F$ denotes the Frobenius norm. Note that the minimizer to \eqref{eqn:operator_learning} is given by $\Theta^* = Y K^{-1}_{\otimes}$ if $K_{\otimes}$ is invertible. Let $\mathbf{k}^i_j = \mathbf{k}_{\otimes}(\mathbf{u}_j,x_i)$ then $K_{\otimes} = \begin{bmatrix} \mathbf{k}^1_1 & \dots & \mathbf{k}^1_{T_u} & \dots & \mathbf{k}^{T_x}_{1} & \dots & \mathbf{k}^{T_x}_{T_u} \end{bmatrix}$. Note that the initial conditions and input sequences used to construct the columns of $K_{\otimes}$ are related to the output trajectories corresponding to the columns of $\mathbf{Y}$. To formulate the prototype  kernelized operator based DeePC (KerODeePC) problem, at time $k \in \Nset$ we define:
\begin{align*}
    \mathbf{u}_{[0,N-1]}(k) &:= \col(u_{0|k},\dots,u_{N-1|k}), \\
     \mathbf{y}_{[1,N]}(k) &:= \col(y_{1|k},\dots,y_{N|k}),
\end{align*}
where for a variable $a$ the notation $a_{i|k}$ is the predicted value of $a_{k+i}$ at time instant $k$. 
\begin{remark}
    Due to the properties of the Kronecker product it holds that $K_{\otimes}^{-1} = K_u^{-1} \otimes K_x^{-1}$. This significantly reduces the computational complexity of inverting $K_{\otimes}$. I.e. $K_{\otimes}$ is guaranteed to be invertible if and only if $K_u$ and $K_x$ are invertible.
\end{remark}

Next, we present the KerODeePC formulation.
\begin{problem} (KerODeePC Problem) \label{prob:KDPC}
    \begin{subequations}
\begin{align}
\min_{\Xi} &\quad l_N(y_{N|k}) + \sum_{i=0}^{N-1}l(y_{i|k},u_{i|k}) + \lambda l_g(\mathbf{g}_k),  \nonumber \\
&\text{subject to constraints:} \nonumber \\
&K_{\otimes} \mathbf{g}_k = \mathbf{k}_{\otimes}(\mathbf{u}_{[0,N-1]}(k),x_k) \label{eqn:DeePC1b}, \\
&Y \mathbf{g}_k = \mathbf{y}_{[1,N]}(k), \label{eqn:DeePC1c}\\
&(\mathbf{y}_{[1,N]}(k), \mathbf{u}_{[0,N-1]}(k)) \in \Yset^N \times \Uset^N,\label{eqn:DeePC1d}
\end{align}
\end{subequations}
\end{problem}
where $\Xi = \{\mathbf{u}_{[0,N-1]}(k),\mathbf{y}_{[1,N]}(k),\mathbf{g}_k\}$. For suitable references $(r^y_{i|k}, r^u_{i|k})$, we define the terminal cost $l_N(y_{N|k}):=(y_{N|k}-r^y_{N|k})^\top P(y_{N|k}-r^y_{N|k})$, and the stage cost $l(y_{i|k},u_{i|k}):=(y_{i|k}-r^y_{i|k})^\top Q(y_{i|k}-r^y_{i|k})+ (u_{i|k}-r^u_{i|k})^\top R (u_{i|k}-r^u_{i|k})$. $l_g(\mathbf{g}_k) := \|g_k\|_2^2$ is a regularization cost. We assume that $P \succ 0$, $Q \succ 0$, $R \succ 0$ and $\lambda>0$. 

Note that the gram matrix $K_{\otimes}$ and the data matrix $\mathbf{Y}$ are constructed from the offline collected data matrices $\mathbf{U}$, $\mathbf{Y}$ and $\mathbf{X}_0$. In the next subsection we discuss a systematic method for generating these data matrices.

\subsection{Data-generation for operator learning}\label{sec:data-generation}
For simplifying the selection of centroids in the data generation algorithm we adopt the following assumption.
\begin{assumption}\label{ass:1}
The state constraints and input constraints sets \( \mathbb{X}, \mathbb{U} \) are polytopes that contain the equilibrium points or references of interest in their interior.
\end{assumption}

If the constraints sets are polytopes (or boxes), than quasi-random methods can be used to initialized the centroids, e.g., such as Halton sequences. Also, for simplicity of exposition, we assume that all generated trajectories remain within the state constraints set $\mathbb{X}$. Ideally, we aim to capture the system's dynamics across all potential initial conditions within this set. However, this leads to computational complexity. Moreover, many systems are not globally asymptotically stable which complicates the selection of a set of initial conditions. With this in mind, we propose a systematic selection method for initial conditions in Algorithm~\ref{alg:1}.
\begin{algorithm}
\caption{Initial condition generation by k-means}\label{alg:1}
\begin{algorithmic}[1]
\State \textbf{Input:} $T_{u,ini}$, $T_x$, $x_0$
\State Define an input signal: $\Tilde{\mathbf{u}} = \{u_1, u_2, \dots, u_{T_{u,ini}}\}$
\State Apply $\Tilde{\mathbf{u}}$ to \eqref{eqn:state_space} and $x_0$ and store $\Tilde{\mathbf{X}} = \begin{bmatrix}
        x_1 & \dots  & x_{T_{u,ini}}
        \end{bmatrix}$
\State \textbf{Initialize} \(k\) cluster centers $\mathbf{\mu}  = \begin{bmatrix}
        \mu_1 &  \mu_2 & \dots & \mu_{T_x}
        \end{bmatrix}$ (centroids) in $\mathbb{X}$
\Repeat
    \State Assign points to clusters $\mathcal{C}_i$ given by 
    $$
    \mathcal{C}_i = \left\{ x_j \in \Tilde{\mathbf{X}} : d(x_j,\mu_i)  \leq d(x_j,\mu_l)  \; \forall \, l \in [1,T_x]\right\}
    $$
    \State Update centroids as mean of each cluster
\Until{Cluster assignments do not change or max iterations reached}
\State \textbf{Output:} Set of initial conditions $\mathbf{X}_0 = C$
\end{algorithmic}
\end{algorithm}

If all elements in the state sequences $\Tilde{\mathbf{X}}$ lie within the constraint-admissible set of states $\mathbb{X}$, then the resulting initial condition set $\mathbf{X}_0$ will also remain within $\mathbb{X}$. After generating the sequences $\Tilde{\mathbf{X}}$, the selection of initial conditions $\mathbf{X}_0$, as outlined in line $4-9$ of Algorithm~\ref{alg:1}, can be achieved by applying MATLAB's \texttt{kmeans} function. Note that different methods exist for step $4-9$. For example the centroids can be initialized uniformly in $\mathbb{X}$, as distance measure one can choose $d(x_i,\mu_j) = \|x_1 - \mu_j \|_2^2$ and different stopping criteria can be used. 
Once a set of initial conditions $\mathbf{X}_0$ is obtained we can construct the matrices $K_{\otimes}$ and $\mathbf{Y}$ in Problem~\ref{prob:KDPC}. In this regard, consider a set of input sequences $\{\mathbf{u}_i : i = 1, 2, \dots, T_u\}$ with $\mathbf{u}_i \in \mathcal{U} $ and set of initial states $\{x_j : j = 1, 2, \dots, T_x \}$ with $x_j \in \mathbf{X}_0$ we generate a corresponding set of output trajectories as summarized in Algorithm~\ref{alg:2}.
\begin{algorithm}
\caption{Data-generation for operator learning}\label{alg:2}
\begin{algorithmic}[1]
\State \textbf{Input:} $\mathbf{X}_0$
\State Define the set of input signals: $\{\mathbf{u}_1, \mathbf{u}_2, \dots, \mathbf{u}_{T_u}\}$, where $\mathbf{u}_j = \{u_0^j, u_1^j, \dots, u_{N-1}^j\}$
\For{each initial condition $x^i$ from $i=1$ to $T_x$}
    \For{each input signal $\mathbf{u}_j$ from $j=1$ to $T_u$}
        \State Simulate the system \eqref{eqn:state_space} and generate output
        \State signals $\mathbf{y}_j^i = \{y_1^{i,j}, y_2^{i,j}, \dots, y_{N}^{i,j}\}$
    \EndFor
\EndFor
\end{algorithmic}
\end{algorithm}

We collect the initial conditions, input sequences and output sequences in the data matrices:
\begin{align*}
    \mathbf{X}_0 &:= \begin{bmatrix} x_1 & \dots & x_{T_x} \end{bmatrix}, \quad  \mathbf{U} := \begin{bmatrix} \mathbf{u}_1 & \dots & \mathbf{u}_{T_u} \end{bmatrix}, \\
    \mathbf{Y} &:= \begin{bmatrix}\mathbf{y}^1_1 & \dots & \mathbf{y}^1_{T_u} & \dots,\mathbf{y}^{T_x}_{1} & \dots & \mathbf{y}^{T_x}_{T_u} \end{bmatrix}.
\end{align*}
The corresponding $K_x$ and $K_u$ matrices are obtained by applying the feature maps $\mathbf{k}_x$ and $\mathbf{k}_u$ to the columns of the matrices $\mathbf{X}_0$ and $\mathbf{U}$ respectively, i.e.,
\begin{align*}
    K_x &= \begin{bmatrix}
    \mathbf{k}_x(x_1) & \dots & \mathbf{k}_x(x_{T_x})
\end{bmatrix}, \\
K_u &= \begin{bmatrix}
    \mathbf{k}_u(\mathbf{u}_1) & \dots & \mathbf{k}_u(\mathbf{u}_{T_u})
\end{bmatrix}.
\end{align*}
Note that the dimension of $K_{\otimes}$ becomes very large for large $T_x$ and $T_u$. In this regards we present an alternative formulation of Problem~\ref{prob:KDPC} in the next section.

\section{Computationally Efficient KerODeePC}\label{sec:reduction}
The dimension of $K_{\otimes}\in \mathbb{R}^{T_x T_u \times T_x T_u}$ becomes very large for large $T_u$ and $T_x$. As a result, implementation of Problem~\ref{prob:KDPC} becomes challenging. To address this, next, we exploit the structure of the product kernel function $\mathbf{k}_{\otimes}(\mathbf{u},x)$ to rewrite Problem~\ref{prob:KDPC} such that the implementation becomes much more efficient, as follows. 
\begin{lemma}\label{lemma_1} 
Let \(\mathbf{k}_x(x) \neq \mathbf{0}\) for all \(x \in \mathcal{X}\), and suppose \(K_x\) and \(K_u\) are full-rank matrices. Then, the system of equations \eqref{eq:willems} can be equivalently re-written as the reduced system of equations: 
\begin{subequations} \label{eq:willems_reduced} \begin{align} \Omega(x) \mathbf{g} &= \mathbf{k}_u(\mathbf{u}), \label{eq:willems_reduceda} \\ \mathbf{Y} \mathbf{g} &= \mathbf{y}, \label{eq:willems_reducedb} \end{align} \end{subequations} 
where \(\Omega(x) := \left( I_{T_u} \otimes \frac{\mathbf{k}^\top_{x}(x)}{\| \mathbf{k}^\top_{x}(x) \|_2^2} \right) K_{\otimes} \in \mathbb{R}^{T_u \times T_u T_x}\). \end{lemma}
\begin{proof}
Due to the properties of the Kronecker product we can rewrite \eqref{eqn:willemsa} as follows:
    \begin{align*}
       K_{\otimes} \mathbf{g} &= \mathbf{k}_{\otimes}(\mathbf{u},x) \\
        &= \mathbf{k}_u(\mathbf{u}) \otimes \mathbf{k}_x(x) \\
        &= \left( I_{T_u} \otimes \mathbf{k}_x(\mathbf{x}) \right) \left( \mathbf{k}_u(\mathbf{u}) \otimes I_{1} \right).
    \end{align*}
Here we used the property that for two matrices $A\in\mathbb{R}^{m_1\times n_1}$ and $B\in\mathbb{R}^{m_2\times n_2}$ it holds that $A \otimes B = (I_{m_1}\otimes B)(A \otimes I_{n_2})$. Since $I_1 = 1$ this gives the following expression:
\begin{align}\label{eqn:proof1}
         K_{\otimes} \mathbf{g}  = \left( I_{T_u} \otimes \mathbf{k}_x(\mathbf{x}) \right) \mathbf{k}_u(\mathbf{u}) .
    \end{align}
 Moreover, for the Kronecker product it holds that $(A \otimes B)^{\dag} = A^{\dag} \otimes B^{\dag}$ and ${\displaystyle \operatorname {rank} (A \otimes B )=\operatorname {rank} A \,\operatorname {rank} B}$. Since $\operatorname {rank} I_{T_u} = T_u$ it holds that $\left( I_{T_u} \otimes \mathbf{k}_x(x) \right)$ is full column rank if $\mathbf{k}_x(x)\neq\mathbf{0}$. This implies that the pseudo-inverse exists and is given by:  
\begin{align*}
   \left( I_{T_u} \otimes \mathbf{k}_x(\mathbf{x}) \right)^{\dag} = \left( I_{T_u} \otimes \frac{\mathbf{k}^\top_{x}(x)}{\| \mathbf{k}_{x}(x) \|_2^2} \right).
\end{align*}
Multiplying both sides of \eqref{eqn:proof1} from the left by this expression gives \eqref{eq:willems_reduceda}, which completes the proof.
\end{proof}

%\begin{remark}\label{remark:nonuniqueness}
%    Note the required number of equality constraints to implement \eqref{eq:willems_reduced} is reduced from $T_x T_u$ to $T_u$ with respect to \eqref{eq:willems}. Due to this the solutions for $\mathbf{g}$ in \eqref{eq:willems_reduced} is no longer unique.
%\end{remark}

By using the formulation in \eqref{eq:willems_reduced} as a prediction model for Problem~\ref{prob:KDPC}, the number of equality constraints reduces significantly from \(T_x T_u\) to \(T_u\) after substituting the initial condition \(x\). However, the decision vector $\mathbf{g}$ remains in \(\mathbb{R}^{T_x T_u}\), which continues to pose implementation challenges. In \cite{huang2023robust}, an iterative approach was introduced, solving $\mathbf{g}$ through a QP and \(\mathbf{u}\) through a nonlinear optimization. Here, we propose a one-shot solution to efficiently handle the full KerODeePC problem directly, based on an equivalent predictor representation inspired by \cite{lazar2024basis}. To this end, we will use the following notion of model equivalence. 

\begin{definition}\label{def:equivalence}
For the system \eqref{eqn:state_space} consider two models $\mathbf{y}^{M_1} = G_{M_1}(\mathbf{u},x)$ and $\mathbf{y}^{M_2} = G_{M_2}(\mathbf{u},x)$. These models are called \emph{equivalent} if for every $x\in\mathcal{X}$ and every $\mathbf{u} \in \mathcal{U}$ it holds that
\begin{align}
    \| \mathbf{y}^{M_1} - \mathbf{y}  \| = \| \mathbf{y}^{M_2} - \mathbf{y}  \|,
\end{align}
where $\mathbf{y}$ is the true system \eqref{eqn:state_space} output obtained form initializing \eqref{eqn:state_space} at $x$ and applying the sequences $\mathbf{u}$ recursively. 
\end{definition}

Notice that one way to establish model equivalence is to show that $\mathbf{y}^{M_1} = \mathbf{y}^{M_2}$ for all $x\in\mathcal{X}$ and every $\mathbf{u} \in \mathcal{U}$. 
\begin{lemma}\label{lemma:reduction}
    Consider the system of equations \eqref{eq:willems} and the reduced system of equations \eqref{eq:willems_reduced} defined using the same set of data
$\{\mathbf{U}, \mathbf{Y}, \mathbf{X}_0 \}$ generated using system \eqref{eqn:state_space}. Suppose $K_{\otimes} \succ 0$. Then, by Definition~\ref{def:equivalence}, the prediction models \eqref{eq:willems} and \eqref{eq:willems_reduced} are equivalent models for system \eqref{eqn:state_space} if and only if $\mathbf{Y}\Hat{\mathbf{g}}=\mathbf{0}$ for all $\Hat{\mathbf{g}}\in\mathcal{N}\left( \Omega(x)\right)$, for all $x \in \mathcal{X}$, where $\mathcal{N}(\cdot)$ denotes the null space.
\end{lemma}
\begin{proof}
    Let $\mathbf{y}^{\Omega}$ denote the solution for $\mathbf{y}$ in \eqref{eq:willems_reduced} generated by the data sets $\{\mathbf{U}, \mathbf{Y}, \mathbf{X}_0 \}$. This solution is not unique, but given by the set 
    \begin{align*}
       \mathbf{y}^{\Omega} \in \left\{ \mathbf{Y} \left(\Omega^{\dag}(x) \mathbf{k}_u(\mathbf{u}) + \Hat{\mathbf{g}} \right) : \Hat{\mathbf{g}} \in \mathcal{N}(\Omega(x)) \right\}.
        \end{align*}
Above we used the fact that the set of all $\bg$ that satisfy \eqref{eq:willems_reduceda} can be explicitly characterized as
\[\bg\in\{\Omega^{\dag}(x) \mathbf{k}_u(\mathbf{u}) + \Hat{\mathbf{g}} : \Hat{\mathbf{g}} \in \mathcal{N}(\Omega(x))\}.\]
Above, $\Omega^{\dag}(x)$, which always exists if $K_\otimes$ is full rank, is analytically given by
    $$
        \Omega^{\dag}(x) = K^{-1}_{\otimes} (I_{T_u} \otimes \mathbf{k}_{x}(x_k)).
    $$
    This is true because by assumption $K_{\otimes} \succ 0$, which implies that $K_{\otimes}^{-1} = K_{\otimes}^{\dag}$. Let $\mathbf{y}^{\otimes}$ denote the solution for $\mathbf{y}$ in \eqref{eq:willems} generated using the same data sets, then for all $x \in \mathcal{X}$ and for all $\mathbf{u} \in \mathcal{U}$ the solution is \emph{uniquely} given by:
    \begin{align*}
    \mathbf{y}^{\otimes} :=&  \mathbf{Y} K^{-1}_{\otimes} \mathbf{k}_{u}(\mathbf{u}) \otimes \mathbf{k}_{x}(x) \\
    =& \mathbf{Y} K^{-1}_{\otimes} (I_{T_u} \otimes \mathbf{k}_{x}(x))\mathbf{k}_{u}(\mathbf{u}) \\
    =& \mathbf{Y} \Omega^{\dag}(x) \mathbf{k}_{u}(\mathbf{u}).
\end{align*}
From this it can be concluded that $\mathbf{y}^{\Omega} =  \mathbf{y}^{\otimes}$ if and only if $\mathbf{Y}
\Hat{\mathbf{g}} = \mathbf{0}$ for all $\Hat{\mathbf{g}} \in \mathcal{N}(\Omega(x))$.
\end{proof}

From Lemma~\ref{lemma:reduction} it follows that the set of solutions for equations \eqref{eq:willems_reduceda} and \eqref{eq:willems_reducedb} are equivalent to the set of solutions for:  
 \begin{subequations}
\begin{align}
    \Omega(x) \Hat{\mathbf{g}} &= \mathbf{0}, \\ 
     \mathbf{Y}\left( \Omega^{\dag}(x) \mathbf{k}_{u}(\mathbf{u}) + \Hat{\mathbf{g}} \right)& = \mathbf{y} .
\end{align}
 \end{subequations}
Finally, to reduce the dimension of the vector of variables $\Hat{\mathbf{g}}\in \mathbb{R}^{T_x T_u}$, we adopt a projection on the output space, i.e., we aim to substitute $\Hat{\mathbf{g}}$ with $\Tilde{\mathbf{g}}\in\mathbb{R}^{Np}$, which are related by $\Hat{\mathbf{g}}=Y^{\dag}\Tilde{\mathbf{g}}$. This yields the following computationally efficient KerODeePC problem formulation.

\begin{problem} (Efficient KerODeePC Problem) \label{prob:KDPC-r}
    \begin{subequations}
\begin{align}
\min_{\Xi} &\quad l_N(y_{N|k}) + \sum_{i=0}^{N-1}l(y_{i|k},u_{i|k}) + \lambda l_g(\Tilde{\mathbf{g}}_k) , \label{eqn:DeePC1a-r} \\
&\text{subject to constraints:} \nonumber \\
&\Omega(x_k) \mathbf{Y}^{\dag}\Tilde{\mathbf{g}}_k = \mathbf{0} \label{eqn:DeePC1b-r}, \\
&\mathbf{Y} \Omega^{\dag}(x_k) \mathbf{k}_{u}(\mathbf{u}_{[0,N-1]}(k))) + \Tilde{\mathbf{g}}_k =\mathbf{y}_{[1,N]}(k), \label{eqn:DeePC1c-r}\\
&(\mathbf{y}_{[1,N]}(k), \mathbf{u}_{[0,N-1]}(k)) \in \Yset^N \times \Uset^N,\label{eqn:DeePC1d-r}
\end{align}
\end{subequations}
\end{problem}

In Problem~\ref{prob:KDPC-r} the free variables are now mapped into the predicted output space of dimension $pN$, which is
typically much smaller than $T_x T_u$. The price to pay is less free variables for optimizing the bias/variance trade–off and it is necessary that the output data matrix $\mathbf{Y}$ has full row-rank. Let $\mathbf{y}^{\Omega_e}_{[1,N]}(k)$ denote the solution to Problem~\ref{prob:KDPC-r} and let $\mathbf{y}^{\otimes}_{[1,N]}(k)$ denote the solutions to Problem~\ref{prob:KDPC} at time instant $k$. In the formulation of Problem~\ref{prob:KDPC-r}, when $\lambda \rightarrow \infty$, we have that $\Tilde{\mathbf{g}}_k \rightarrow 0$ and thus $\Hat{\mathbf{g}}_k \rightarrow 0$ by Lemma~\ref{lemma_1} this implies that $\mathbf{y}^{\Omega_e}_{[1,N]}(k) \rightarrow \mathbf{y}^{\otimes}_{[1,N]}(k)$.

\section{Illustrative example}\label{sec:example}
To assess the effectiveness of the developed Efficient KerODeePC scheme defined in Problem~\ref{prob:KDPC-r}, we consider the discretized Van der Pol oscillator, i.e.,
\begin{equation}\label{eqn:example}
    \begin{split}
        \begin{bmatrix}
            x_1(k+1) \\
            x_2(k+1)     
        \end{bmatrix} &= 
        \begin{bmatrix}
            1 & T_s \\
            -T_s & 1    
        \end{bmatrix}  \begin{bmatrix}
            x_1(k) \\
            x_2(k)     
        \end{bmatrix} +  \begin{bmatrix}
            0 \\
            T_s    
        \end{bmatrix} u(k) \\
         &+ 
        \begin{bmatrix}
            0 \\
            T_s \mu (1-x_1^2(k))x_2(k)
        \end{bmatrix}, \\
        y(k) &= x_1(k),
    \end{split}
\end{equation}
where $u(k)$ and $y(k)$ are the control input and output at time instant $k$, while $\mu =1$ and the sampling time is $T_s = 1/10$s. We apply Algorithm~\ref{alg:1} with $T_{u,ini}=100$ and $T_x = 20$. For the input sequence $\Tilde{\mathbf{u}}$, we generate a sine wave input using the \texttt{idinput} function. The parameters are set as: frequency band $[0, 1]$, input range $[-1, 1]$, and sine characteristics specified by $\text{SineData} = [25, 40, 1]$. The initial conditions are shown in \figurename{\ref{fig:initial}}.

% \begin{figure}[b!]
%   \centering
% \includegraphics[scale=0.85,trim=0.0cm 0.0cm 0.0cm 0.0cm,clip]{x0_smi.eps}\vspace{-.2cm}
%   \caption{Set of 10 initial conditions $\mathbf{X}_0$ generated by applying Algorithm~\ref{alg:1}.}
%   \label{fig:X0}
% \end{figure}
% \begin{figure}[t!]
%   \centering
% \includegraphics[scale=0.85,trim=0.0cm 0.0cm 0.0cm 0.0cm,clip]{kernel_control_compare.eps}\vspace{-.2cm}
%   \caption{Reference tracking for product/stacked kernel for $T_x = 10$ and $T_u = 20$.}
%   \label{fig:tracking}
% \end{figure}
% \begin{figure}[b!]
%   \centering
% \includegraphics[scale=0.85,trim=0.0cm 0.0cm 0.0cm 0.0cm,clip]{x0_smi2.eps}\vspace{-.2cm}
%   \caption{Set of 200 initial conditions $\mathbf{X}_0$ generated by applying Algorithm~\ref{alg:1}.}
%   \label{fig:X02}
% \end{figure}
For generating a set of input sequences $\mathbf{U}$ we use $T_{u}= 20$ and we use the \texttt{idinput} function with the same settings. The prediction horizon is set to $N=10$, this gives the input data matrix
$$
\mathbf{U} = \begin{bmatrix}
    u(1) & u(2) & \dots & u(T_u) \\
    \vdots & \vdots & \dots & \vdots \\
    u(N) & u(N+1) & \dots & u(N+T_u)
\end{bmatrix}\in \mathbb{R}^{10\times 20}.
$$
The output matrix $\mathbf{Y}$ is constructed by using Algorithm~\ref{alg:2} for the Van der Pol oscillator \eqref{eqn:example}. We use Gaussian kernels, i.e.,
$$
k_u(u_i,u_j) = e^{ -\frac{1}{\sigma_u^2}\|u_i - u_j \|_2^2}, \quad k_x(x_i,x_j) = e^{- \frac{1}{\sigma_x^2}\|x_i - x_j \|_2^2},
$$
with $\sigma_u = 50$ and $\sigma_x = 3$. Using the data matrices $\mathbf{X}_0$, $\mathbf{Y}$ and $\mathbf{U}$ we construct the product kernel gram matrix $K_{\otimes}\in\mathbb{R}^{400\times 400}$ and the corresponding kernel vector $\mathbf{k}_{\otimes}(\mathbf{u},x)\in\mathbb{R}^{400}$. We use these data matrices to implement Problem~\ref{prob:KDPC-r}. In order to validate the proposed method we compare with a stacked kernel formulation corresponding to the formulations in \cite{huang2023robust,maddalena2021kpc}. To this end, define the stacked variable \(\mathbf{z}_{i} = \operatorname{col}(x_i, \col(u_i,\dots,u_{i+N-1}))\) resulting in the kernel vector and Gram matrix:
    \begin{subequations}
\begin{align}
    &\mathbf{k}_{z}(\mathbf{z}) := 
\begin{bmatrix}
    k(\mathbf{z}, \mathbf{z}_{1}) \\
    \vdots \\
    k(\mathbf{z}, \mathbf{z}_{T})
\end{bmatrix} \in \mathbb{R}^{T}, \\
&K_z := 
\begin{bmatrix}
    \mathbf{k}_{z}(\mathbf{z}_{1}) & \dots & \mathbf{k}_{z}(\mathbf{z}_{T})
\end{bmatrix} \in \mathbb{R}^{T \times T}.
\end{align}
\end{subequations}
In order to analyze the predictive performance of both models we calculate the following multi-step input-output identified predictors:
    \begin{align}\label{eqn:prod_pm}
    \mathbf{y}^{product}_{[1,N]}(k) &:= \mathbf{Y}K_{\otimes}^{-1}  \mathbf{k}_{\otimes}(\mathbf{u}_{[0,N-1]}(k),x_k). \\
    \mathbf{y}^{stacked}_{[1,N]}(k) &:= \mathbf{Y}K_{z}^{-1}  \mathbf{k}_{z}(\col(x_k,\mathbf{u}_{[0,N-1]}(k))).
\end{align}
The input signal is again generated using \texttt{idinput} with the same settings as for the product kernel formulation and the same data length $T=T_x T_u = 400$. This signal is then applied to the Van der Pol oscillator \eqref{eqn:example}. In order to have a fair comparison we use as kernel function $k(\mathbf{z}_j,\mathbf{z}_i) = e^{-\sum_{i=1}^{N+n}\frac{1}{\sigma_i^2}(\mathbf{z}_j-\mathbf{z}_i)^2}$, where $\sigma_i=3$ for $i=1:n$ and $\sigma_i=50$ for $i=n+1:n+N$, i.e. the same weightings as for the product kernel formulation. Due to the stacking of initial conditions and input sequences it is no longer possible to decouple the equations. As a workaround, we implemented the stacked kernel formulation as in Problem~\ref{prob:KDPC} where $K_{\otimes}$ and $\mathbf{k}_{\otimes}$ are replaced with $K_{z}$ and $\mathbf{k}_{z}$. 

The problems for both kernel formulations are solved using the \texttt{fmincon} solver. The stacked kernel formulation requires much more time, see Table~\ref{tab:comp}. For $T_x=20$, $T_u=20$ and $T=400$, the required computation time per control action is $0.3319$s for the product kernel formulation compared with $33.8668$s for the stacked kernel formulation. Once we increase the data to $T_x=200$, $T_u=50$ and $T=10000$ it is no longer possible to apply the stacked kernel formulation for control while the product kernel formulation can still be solved relatively fast $1.3508$s, which is still significantly faster than $33.8668$s for the stacked kernel formulation with a much smaller data size $T=400$. See \figurename{\ref{fig:initial}} for the set of initial conditions for $T_x=200$.
\begin{table}[t!]
\centering
\caption{Computational Time Comparison}
\begin{tabular}{|c||c|c|c|c|}
 \hline
 \multirow{2}{*}{\textbf{Comp. Time} } &
      \multicolumn{2}{c|}{$\mathbf{T = 400}$} &
      \multicolumn{2}{c|}{$\mathbf{T = 10000}$}  \\
      & \textbf{Stacked} & \textbf{Product} & \textbf{Stacked} & \textbf{Product}  \\
 \hline\hline
Gram inversion & $0.0042$s & \textcolor{green_dark}{$0.0008$s} & $18.3382$s & \textcolor{green_dark}{$0.0356$s} \\ 
 \hline
Gram construction & $0.0401$s & \textcolor{green_dark}{$0.0060$s} & $425.2443$s & \textcolor{green_dark}{$0.0552$s} \\ 
 \hline
Control action & $33.8668$s & \textcolor{green_dark}{$0.3319$s} & \textcolor{red}{X} & \textcolor{green_dark}{$1.3508$s} \\ 
 \hline
\end{tabular}
\label{tab:comp}
\end{table}
\begin{table}[t!]
\centering
\caption{Predictive Control Performance Comparison}
\begin{tabular}{|c||c|c|c|c|}
 \hline
 \multirow{2}{*}{\textbf{ } } &
      \multicolumn{2}{c|}{$\mathbf{T = 400}$} &
      \multicolumn{2}{c|}{$\mathbf{T = 10000}$}  \\
      & \textbf{Stacked} & \textbf{Product} & \textbf{Stacked} & \textbf{Product}  \\
 \hline\hline
 Mean tracking error & $0.0994$ & \textcolor{green_dark}{$0.0917$}  & \textcolor{red}{X}  & \textcolor{green_dark}{$0.0835$}  \\ 
 \hline
Mean prediction error & $0.0216$ & \textcolor{green_dark}{$0.0183$} & \textcolor{red}{X} & \textcolor{green_dark}{$0.0157$} \\ 
 \hline
\end{tabular}
\label{tab:performance}
\end{table}
\begin{figure}[b!]
  \centering
\includegraphics[scale=0.78,trim=0.0cm 0.0cm 0.0cm 0.0cm,clip]{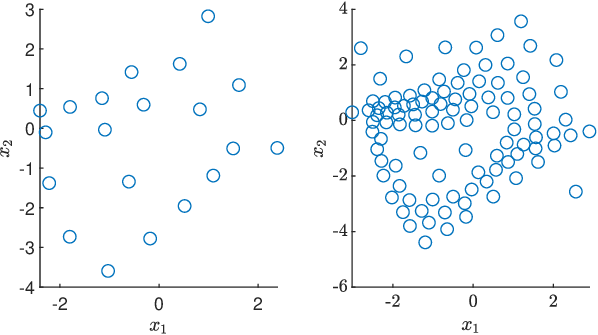}\vspace{-.2cm}
  \caption{Initial conditions for $T_x = 20$ and $T_x = 200$ generated using Algorithm~\ref{alg:1}.}
  \label{fig:initial}
\end{figure}
\begin{figure}[t!]
  \centering
\includegraphics[scale=0.78,trim=0.0cm 0.0cm 0.0cm 0.0cm,clip]{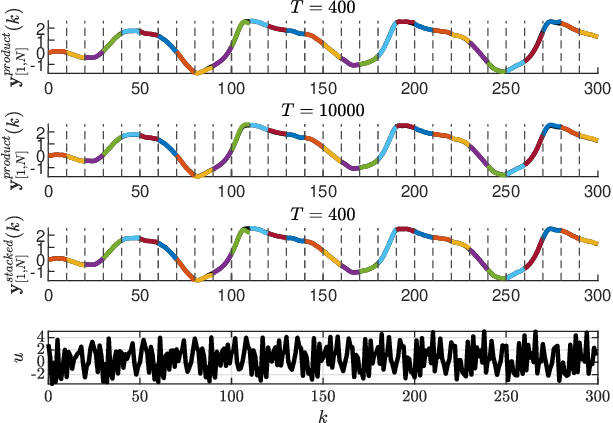}\vspace{-.2cm}
  \caption{Comparison of $N=10$ step-ahead predictions (colored lines) from the stacked/product kernel prediction models against the true system trajectory (black line) for various initial conditions.}
  \label{fig:test}
\end{figure}
\begin{figure}[b!]
  \centering
\includegraphics[scale=0.78,trim=0.0cm 0.0cm 0.0cm 0.0cm,clip]{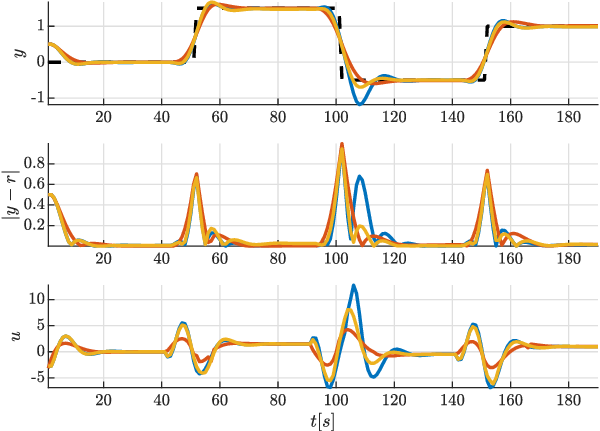}\vspace{-.2cm}
  \caption{Reference tracking for product kernel, $T=400$ (\textcolor{Matlab_blue}{\textbf{---}}), stacked kernel, $T=400$ (\textcolor{Matlab_orange}{\textbf{---}}), product kernel, $T=10000$ (\textcolor{Matlab_yellow}{\textbf{---}}).}
  \label{fig:tracking}
\end{figure}
\begin{figure}[b!]
  \centering
\includegraphics[scale=0.78,trim=0.0cm 0.0cm 0.0cm 0.0cm,clip]{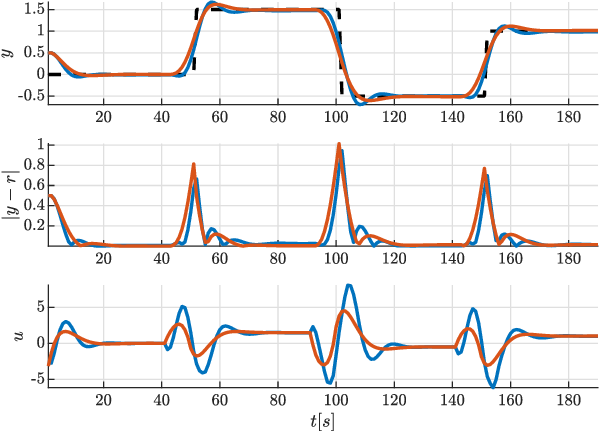}\vspace{-.2cm}
  \caption{Reference tracking for product kernel, $T=10000$ (\textcolor{Matlab_blue}{\textbf{---}}), compared with NMPC (\textcolor{Matlab_orange}{\textbf{---}}).}
  \label{fig:tracking2}
\end{figure}
Note also that the inverse calculation and construction time for the gram matrix are much longer for the stacked kernel than for the product kernel formulation, see Table~\ref{tab:comp}. While the computation times are much shorter for the product kernel, the prediction error of both methods is comparable when the same amount of data is used, see \figurename{\ref{fig:test}}. The average prediction error for all methods is shown in and Table~\ref{tab:performance}. The tracking performance of all methods is compared by tracking a piecewise constant reference signal, see \figurename{\ref{fig:tracking}}. The tracking performance of the product kernel formulations is on average slightly better than the performance of the stacked kernel formulation as shown in Table~\ref{tab:performance}. However, in \figurename{\ref{fig:tracking}} a higher overshoot is visible after a large step in the reference signal. Adding more data improves the performance such that the overshoot is removed. A comparison for the product kernel formulation with $T_x=200$ and $T_u=50$ compared with NMPC is shown in \figurename{\ref{fig:tracking2}}. The product kernel has a tracking performance that is comparable with NMPC. 

\section{Conclusions}
In this work, we presented a novel kernel based operator approach for nonlinear data-enabled predictive control using the product kernel operator learning framework introduced in \cite{Lazar2024productkernel}. The resulting KerODeePC formulation was rendered computationally efficient, particularly for large datasets, by exploiting the product kernel structure. Through numerical examples, we demonstrated that the required computation time per sampling instance was reduced by a factor 30 compared with existing methods when the exact same data set is used. Due to this computational advantage the proposed method allows for much larger data sets yielding superior performance. 

Future work will deal with researching methods for reducing the product kernel dimension without loosing learning accuracy, e.g., by exploiting kernel principal component analysis or sparse feature selection. 

% Generated by IEEEtran.bst, version: 1.14 (2015/08/26)

\end{document}